\documentclass[final]{dmtcs-episciences}

\usepackage{tikz}
\usetikzlibrary{positioning}

\usepackage{amsmath,amssymb}
\usepackage{multirow}
\usepackage{float}
\usepackage[numbers,sort&compress]{natbib}

\numberwithin{equation}{section}

\newtheorem{thm}{Theorem}[section]
\newtheorem{lemma}{Lemma}[section]

\DeclareMathOperator{\red}{red} 
\DeclareMathOperator{\asc}{asc}

\title[Length-4 pattern avoidance in ascent sequences]
{Simultaneous avoidance of length-4 patterns in ascent sequences}

\author[Qi Liu, Sergey Kitaev, Philip B. Zhang]
{Qi Liu\affiliationmark{1}
  \and Sergey Kitaev\affiliationmark{2}
  \and Philip B. Zhang\affiliationmark{1}}

\affiliation{
College of Mathematical Sciences \& Institute of Mathematics
and Interdisciplinary Sciences, Tianjin Normal University,
Tianjin, P. R. China\\
Department of Mathematics and Statistics,
University of Strathclyde, Glasgow, UK}


\hypersetup{
  colorlinks=false,
  pdfborder={0 0 1},
  urlbordercolor={0 1 1}
}

\providecommand{\doi}[1]{%
  doi: \begingroup
  \urlstyle{rm}%
  \href{https://doi.org/#1}{\nolinkurl{#1}}%
  \endgroup
}

\begin{document}

\publicationdata{vol. 28:3}{2026}{18}
{10.46298/dmtcs.17977}
{2026-04-09; 2026-04-09; 2026-09-28}
{2026-10-07}

\maketitle

\begin{abstract} \vspace{0.29cm}

Ascent sequences form a central class of combinatorial objects, as they are in bijection with several important families such as {\bf (2+2)}-free posets, Stoimenow matchings, and other Fishburn objects, and are enumerated by the Fishburn numbers. We study pattern avoidance in ascent sequences for the five patterns of length~$4$: $0101$, $0102$, $0112$, $0120$, and $0121$. These patterns arise naturally from recent work on pattern avoidance in related families of Fishburn objects, including Stoimenow matchings and {\bf (2+2)}-free posets. We enumerate ascent sequences avoiding all nonempty subsets of these patterns, with the exception of the sets $\{0120\}$, $\{0121\}$, and $\{0120,0121\}$, for which the enumeration remains open. Our results reveal that the corresponding avoidance classes fall into $16$ Wilf equivalence classes and exhibit a wide range of enumerative behavior, including connections to classical sequences such as the Catalan and Fibonacci numbers, as well as polynomial formulas and rational generating functions; several of the sequences we obtain appear to be new. Our methods combine structural decompositions with generating-tree techniques and, in several cases, rely on reductions to shorter patterns via restricted growth functions. This work contributes to the broader study of pattern avoidance across Fishburn families and highlights further connections between ascent sequences and other combinatorial structures.

\end{abstract}

\section{Introduction}\label{sec:intro}

An \emph{ascent} in an integer sequence $s_1 s_2 \cdots s_m$ is an index
$1 \le j \le m-1$ such that $s_j < s_{j+1}$. We write $\asc(s)$ for the number of ascents in a sequence $s$.
An \emph{ascent sequence} is a sequence of nonnegative integers
$x = x_1 x_2 \cdots x_n$ satisfying $x_1 = 0$ and, for each $i \ge 2$,
\[
x_i \le 1+\asc(x_1x_2\cdots x_{i-1}).
\]
For example, 0101024 is an ascent sequence. Elements of ascent sequences are referred to as letters. The number of ascent sequences is given by the \emph{Fishburn numbers}; this is sequence A022493 in the Online Encyclopedia of Integer Sequences (OEIS)~\cite{OEIS}.

Pattern avoidance for ascent sequences is formulated in analogy with permutation patterns, but allowing repetitions.
Given an integer sequence $w = w_1 w_2 \cdots w_m$, its \emph{reduction}, denoted $\red(w)$, is obtained by replacing the $i$th smallest distinct letter appearing in $w$ by $i-1$. For instance, $\red(2664562) = 0331230$.
A \emph{pattern} is simply a reduced sequence.
An ascent sequence $x = x_1 x_2 \cdots x_n$ \emph{contains} a pattern $p = p_1 p_2 \cdots p_k$ if there exist indices $1 \le i_1 < i_2 < \cdots < i_k \le n$ such that $\red(x_{i_1} x_{i_2} \cdots x_{i_k}) = p$; otherwise, $x$ \emph{avoids} $p$.
For a finite set of patterns $P$, let $\mathcal{A}_n(P)$ denote the set of ascent sequences of length $n$ avoiding every pattern in $P$, and write $a_P(n) = |\mathcal{A}_n(P)|$. Similarly, for permutations, let $S_n(P)$ denote the set of permutations of length $n$ that avoid each pattern in $P$.
For two sets of patterns $P$ and $Q$, we say that $P$ and $Q$ are \emph{Wilf-equivalent} if $a_P(n) = a_Q(n)$ for all $n \geq 1$.

Ascent sequences were introduced in \cite{BMCDK} by Bousquet-M\'elou, Claesson, Dukes, and Kitaev in connection with \textbf{(2+2)}-free posets
(which are equinumerous with interval orders) and have since become a central family in enumerative
combinatorics; see, for example,
\cite{BaxterPudwell2015,CallanMansour2025,CMS2014,
Cerbai2025,CerbaiClaesson2023,CDDDS2013,CCEG2022,
DisFerPinRin,DisantoPerPinRin,DM2020,DP,DRKS,
DuncanSteingrimsson2011,Egge2022,GilWeiner2019,
Gowravaram,KimRou,LF,LKLZ2025,Luce,LKZ2025,
ManSha2014,MansourShattuck,P2016,Y2014}. In particular, the study of pattern avoidance in ascent sequences, initiated systematically by Duncan and Steingrímsson in \cite{DuncanSteingrimsson2011}, has developed in close analogy with the corresponding theory for permutations, beginning with patterns of length $3$ and gradually extending to more complex settings.

Early work focused on avoiding one or more patterns of length $3$. For instance, in \cite{BaxterPudwell2015}, Baxter and Pudwell investigated ascent sequences avoiding pairs of patterns of length $3$, obtaining exact enumerations for $16$ different pairs. Their methods include simple recurrences, generating trees, and bijections to other combinatorial structures such as Dyck paths and pattern-avoiding permutations. Further developments in this direction include the work of Callan and Mansour \cite{CallanMansour2025}, who studied ascent sequences avoiding triples of length-$3$ patterns. They showed that there are $62$ distinct Wilf equivalence classes, using a combination of bijective arguments and generating tree techniques.

The investigation of longer patterns has revealed a richer and more intricate structure. In~\cite{P2016}, Pudwell proved that ascent sequences avoiding $0021$, as well as those avoiding both $201$ and $210$, are enumerated by the binomial convolution of the \emph{ Catalan numbers} $C_n=\frac{1}{n+1}\binom{2n}{n}$. This result provides an enumeration for the pattern $0021$ and extends earlier results on pairs of length-$3$ patterns. Moving to simultaneous avoidance of patterns of length $4$, Callan, Mansour, and Shattuck \cite{CMS2014} classified all pairs $\{u_1u_2u_3u_4, v_1v_2v_3v_4\}$ for which $a_{\{u_1u_2u_3u_4, v_1v_2v_3v_4\}}(n)=C_n$. Moreover, in $5$ out of $8$ cases, they refined this enumeration by showing that the number of such sequences with exactly $m$ ascents is given by the \emph{ Narayana numbers} $N_{n,m+1} = \frac{1}{n}\binom{n}{m+1}\binom{n}{m}, \quad 0 \le m < n$.
They also constructed a bijection between ascent sequences avoiding $0011$ and $0021$ and Dyck paths.

In this paper, we study avoidance of the five patterns of length~$4$:
\[
0101, \quad 0102, \quad 0112, \quad 0120, \quad 0121.
\]
We enumerate $a_P(n)$ for all nonempty subsets
$P \subseteq \{0101, 0102, 0112, 0120, 0121\}$ except for the subsets
$\{0120\}$, $\{0121\}$, and $\{0120,0121\}$.
Our enumerative results are summarized in Table~\ref{tab-results}. They show that these subsets fall into 16 Wilf equivalence classes (grouped together) and exhibit a wide spectrum of behavior, ranging from classical sequences (such as the Catalan and Fibonacci numbers) to polynomial formulas and rational generating functions. In Table~\ref{tab-results}, for each class we record initial values, an OEIS identifier (when available), and an explicit enumeration (either in closed form or via an ordinary generating function). Cases of enumeration sequences that appear to be new are marked ``New'', and entries for which we do not yet have a closed form are marked ``?''.

Our work is part of a broader program on pattern avoidance in Fishburn objects. A systematic study of pattern avoidance in Fishburn permutations was initiated by Gil and Weiner~\cite{GilWeiner2019}. The choice of the above five patterns is motivated by recent developments connecting different families of such objects. In~\cite{LKZ2025}, Lv, Kitaev, and Zhang addressed a problem posed by Bevan et~al.~\cite{Bevan-Cheon-Kitaev} on identifying subclasses of Stoimenow matchings that are enumerated by the Catalan numbers. They obtained five such subclasses, each defined via avoidance of a certain pattern.

A key feature of Fishburn objects is that they are related by a network of
bijections (see, e.g.,~\cite{LKZ2025}). The behavior of pattern avoidance
under these correspondences has also been investigated. In particular,
Cerbai and Claesson~\cite{CerbaiClaesson2023} developed a theory of
transport of patterns between Fishburn permutations and modified ascent
sequences via the Burge transpose. Since their transport is formulated in
terms of modified ascent sequences, it does not directly produce the
avoidance classes of ascent sequences considered here. We therefore study
these classes directly in the setting of ascent sequences.

The five patterns considered in~\cite{LKZ2025} on matchings correspond, via the bijection in \cite{BMCDK}, to five forbidden configurations $\mathcal{F}_1,\ldots,\mathcal{F}_5$ in {\bf (2+2)}-free posets, shown below:

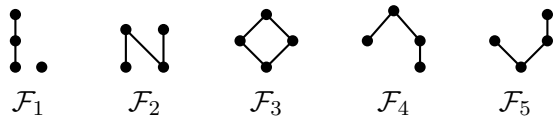
\begin{center}
\begin{tabular}{ccccccccc}
\begin{tikzpicture}[scale=0.5, thick]
    \draw(0,0)--(0,0.7)--(0,1.4);
    \node[circle,fill,inner sep=1.5pt] at (0,0) {};
    \node[circle,fill,inner sep=1.5pt] at (0,0.7) {};
    \node[circle,fill,inner sep=1.5pt] at (0,1.4) {};
    \node[circle,fill,inner sep=1.5pt] at (0.7,0) {};
   
    \node[below=5pt, text centered] at (0.35, 0) {$\mathcal{F}_1$};
\end{tikzpicture}

&
&

\begin{tikzpicture}[scale=0.5, thick]
    \draw(0,0)--(0,1)--(1,0)--(1,1);
    \node[circle,fill,inner sep=1.5pt] at (0,0) {};
    \node[circle,fill,inner sep=1.5pt] at (0,1) {};
    \node[circle,fill,inner sep=1.5pt] at (1,0) {};
    \node[circle,fill,inner sep=1.5pt] at (1,1) {};
    \node[below=5pt, text centered] at (0.5, 0) {$\mathcal{F}_2$};
\end{tikzpicture}

&
&

\begin{tikzpicture}[scale=0.5, thick]
    \draw(0.7,0)--(0,0.7)--(0.7,1.4)--(1.4,0.7)--(0.7,0);
    \node[circle,fill,inner sep=1.5pt] at (0.7,0) {};
    \node[circle,fill,inner sep=1.5pt] at (0,0.7) {};
    \node[circle,fill,inner sep=1.5pt] at (1.4,0.7) {};
    \node[circle,fill,inner sep=1.5pt] at (0.7,1.4) {};
    \node[below=5pt, text centered] at (0.7, 0) {$\mathcal{F}_3$};
\end{tikzpicture}

&
&

\begin{tikzpicture}[scale=0.5, thick]
    \draw(0,0.7)--(0.7,1.5)--(1.4,0.7);
    \draw (1.4,0)--(1.4,0.7);
    \node[circle,fill,inner sep=1.5pt] at (1.4,0) {};
    \node[circle,fill,inner sep=1.5pt] at (0,0.7) {};
    \node[circle,fill,inner sep=1.5pt] at (1.4,0.7) {};
    \node[circle,fill,inner sep=1.5pt] at (0.7,1.5) {};
    \node[below=5pt, text centered] at (0.7, 0) {$\mathcal{F}_4$};
\end{tikzpicture}

&
&
\begin{tikzpicture}[scale=0.5, thick]
    \draw(0.7,0)--(0,0.7);
    \draw(0.7,0)--(1.4,0.7)--(1.4,1.4);
    
    \node[circle,fill,inner sep=1.5pt] at (0.7,0) {};
    \node[circle,fill,inner sep=1.5pt] at (0,0.7) {};
    \node[circle,fill,inner sep=1.5pt] at (1.4,1.4) {};
    \node[circle,fill,inner sep=1.5pt] at (1.4,0.7) {};
   
    \node[below=5pt, text centered] at (0.55, 0) {$\mathcal{F}_5$};
\end{tikzpicture}

\end{tabular}
\end{center}

Recent work by Liu et al.~\cite{LKLZ2025} studies avoidance of subsets of $\{\mathcal{F}_1,\ldots,\mathcal{F}_5\}$ in {\bf (2+2)}-free posets, extending earlier results in the literature. In particular, the avoidance of a single pattern among $\mathcal{F}_1,\ldots,\mathcal{F}_5$ has been well studied: avoiding $\mathcal{F}_1$ yields the Catalan numbers~\cite{KimRou}, while the same is true for $\mathcal{F}_2$~\cite{DisFerPinRin,DisantoPerPinRin}. Further results include the joint avoidance of $\mathcal{F}_1$ and $\mathcal{F}_2$~\cite{DisantoPerPinRin}, as well as the avoidance of $\mathcal{F}_3$ and $\mathcal{F}_5$, which is equivalent to the avoidance of $\mathcal{F}_4$~\cite{Gowravaram}. We also recall that {\bf (2+2,3+1)}-free posets are precisely \emph{ semiorders}, introduced by Luce~\cite{Luce}.

Via the bijection $O$ described in~\cite{LKZ2025}, the posets $\mathcal{F}_1,\ldots,\mathcal{F}_5$ correspond, respectively, to the ascent sequence patterns 0120,  0101,  0112,  0102,  and 0121.
This correspondence motivates our choice of patterns. Our goal is to extend the study of these distinguished avoidance classes to ascent sequences.

The paper is organized as follows. In Section~\ref{sec:prelim}, we discuss the relevant basic notions, establish preliminary lemmas that reduce certain pattern-avoidance problems on ascent sequences to shorter patterns, and give references to known enumerative results for the avoidance of a single pattern. Sections~\ref{avoid-pairs-sec}, \ref{avoid-triple-sec}, and \ref{four-sec} treat the avoidance of pairs, triples, and sets of four and five patterns, respectively, using structural decompositions and, when appropriate, generating-tree arguments. Finally, in Section~\ref{concluding}, we provide concluding remarks.

\begin{table}[htbp]
\centering

\renewcommand{\arraystretch}{1.18}
\setlength{\tabcolsep}{4pt}
\resizebox{\linewidth}{!}{%
\begin{tabular}{|c|c|c|c|c|}
\hline
{\scriptsize Patterns} &
{\scriptsize Sequences for $n\ge 1$} &
{\scriptsize OEIS} &
{\scriptsize Enumeration $(n\ge 2)$} &
{\scriptsize Ref.} \\
\hline
\hline

&&&&\\[-4.5mm]
\scriptsize{$\{0101\}$} &
\scriptsize{$1,2,5,14,42,132,429,1430,\ldots$} &
\scriptsize{A000108} &
\scriptsize{
\begin{tabular}{@{}c@{}}
$\frac{1}{n+1}\binom{2n}{n}$
\end{tabular}} &
\scriptsize{\cite{DuncanSteingrimsson2011}}\\[2mm]
\hline

&&&&\\[-4.5mm]
\scriptsize{$\{0102\}$} &
\multirow{2}{*}{\scriptsize{$1,2,5,14,41,122,365,1094,\ldots$}} &
\multirow{2}{*}{\scriptsize{A007051}} &
\multirow{2}{*}{\scriptsize{
\begin{tabular}{@{}c@{}}
$\dfrac{3^{\,n-1}+1}{2}$
\end{tabular}}} &
\multirow{2}{*}{\scriptsize{\cite{DuncanSteingrimsson2011}}}\\
\cline{1-1}
\scriptsize{$\{0112\}$} &&&&\\
\hline

&&&&\\[-4.5mm]
\scriptsize{$\{0120\}$} &
\scriptsize{$1,2,5,14,42,133,443,1552,5721,22178,90259,\ldots$} &
\scriptsize{New} &
\scriptsize{?} &
\scriptsize{N/A}\\
\hline

&&&&\\[-4.5mm]
\scriptsize{$\{0121\}$} &
\scriptsize{$1,2,5,14,42,133,443,1551,5701,21947,88211,\ldots$} &
\scriptsize{New} &
\scriptsize{?} &
\scriptsize{N/A}\\
\hline
\hline

&&&&\\[-4.5mm]
\scriptsize{$\{0101,0102\}$} &
\multirow{2}{*}{\scriptsize{$1,2,5,13,34,89,233,610,\ldots$}} &
\multirow{2}{*}{\scriptsize{A001519}} &
\multirow{2}{*}{\scriptsize{
\begin{tabular}{@{}c@{}}
$F_{2n-1}$\\
\end{tabular}}} &
\multirow{2}{*}{\scriptsize{Thm~\ref{thm:L2-2}}}\\
\cline{1-1}
\scriptsize{$\{0101,0121\}$} &&&&\\
\hline

&&&&\\[-4.5mm]
\scriptsize{$\{0101,0112\}$} &
\multirow{4}{*}{\scriptsize{$1,2,5,13,33,81,193,449,\ldots$}} &
\multirow{4}{*}{\scriptsize{A005183}} &
\multirow{4}{*}{\scriptsize{
\begin{tabular}{@{}c@{}}
$(n-1)2^{\,n-2}+1$\\
\end{tabular}}} &
\multirow{4}{*}{\scriptsize{Thm~\ref{thm:L2-4}}}\\
\cline{1-1}
\scriptsize{$\{0102,0112\}$} &&&&\\
\cline{1-1}
\scriptsize{$\{0112,0121\}$} &&&&\\
\cline{1-1}
\scriptsize{$\{0102,0120\}$} &&&&\\
\hline

&&&&\\[-4.5mm]
\scriptsize{$\{0101,0120\}$} &
\scriptsize{$1,2,5,13,33,82,202,497,\ldots$} &
\scriptsize{A116703} &
\scriptsize{$\dfrac{(1-x)^3}{1-4x+5x^2-3x^3}$} &
\scriptsize{Thm~\ref{thm:0101-0120}}\\[2mm]
\hline

&&&&\\[-4.5mm]
\scriptsize{$\{0102,0121\}$} &
\scriptsize{$1,2,5,13,32,74,163,347,\ldots$} &
\scriptsize{A116702} &
\scriptsize{
\begin{tabular}{@{}c@{}}
$\frac{3\cdot 2^n-n^2-n-2}{2}$\\[1mm]
\end{tabular}} &
\scriptsize{Thm~\ref{thm:0102-0121}}\\
\hline

&&&&\\[-4.5mm]
\scriptsize{$\{0112,0120\}$} &
\scriptsize{$1,2,5,13,31,67,134,254,466,\ldots$} &
\scriptsize{New} &
\scriptsize{
\begin{tabular}{@{}c@{}}
$2^{\,n-1}+\binom{n+1}{4}$\\[1mm]
\end{tabular}} &
\scriptsize{Thm~\ref{thm:L2-0112,0120}}\\
\hline

&&&&\\[-4.5mm]
\scriptsize{$\{0120,0121\}$} &
\scriptsize{$1,2,5,13,34,90,244,683,1980,5945,18470,\ldots$} &
\scriptsize{New} &
\scriptsize{?} &
\scriptsize{N/A}\\
\hline
\hline

&&&&\\[-4.5mm]
\scriptsize{$\{0101,0102,0112\}$} &
\multirow{6}{*}{\scriptsize{$1,2,5,12,27,58,121,248,\ldots$}} &
\multirow{6}{*}{\scriptsize{A000325}} &
\multirow{6}{*}{\scriptsize{
\begin{tabular}{@{}c@{}}
$2^n-n$\\
\end{tabular}}} &
\multirow{6}{*}{\scriptsize{Thm~\ref{thm:L3-6}}}\\
\cline{1-1}
\scriptsize{$\{0101,0102,0120\}$} &&&&\\
\cline{1-1}
\scriptsize{$\{0101,0102,0121\}$} &&&&\\
\cline{1-1}
\scriptsize{$\{0101,0112,0121\}$} &&&&\\
\cline{1-1}
\scriptsize{$\{0101,0120,0121\}$} &&&&\\
\cline{1-1}
\scriptsize{$\{0102,0120,0121\}$} &&&&\\
\hline

&&&&\\[-4.5mm]
\scriptsize{$\{0101,0112,0120\}$} &
\scriptsize{$1,2,5,12,25,47,82,135,212,\ldots$} &
\scriptsize{A116722} &
\scriptsize{
\begin{tabular}{@{}c@{}}
$\frac{n^4-6n^3+47n^2-114n+120}{24}$\\[1mm]
\end{tabular}} &
\scriptsize{Thm~\ref{thm:L3-1}}\\
\hline

&&&&\\[-4.5mm]
\scriptsize{$\{0102,0112,0120\}$} &
\multirow{3}{*}{\scriptsize{$1,2,5,12,26,52,99,184,340,\ldots$}} &
\multirow{3}{*}{\scriptsize{A116725}} &
\multirow{3}{*}{\scriptsize{
\begin{tabular}{@{}c@{}}
$2^{n-1}+\binom{n}{3}$\\
\end{tabular}}} &
\multirow{3}{*}{\scriptsize{Thm~\ref{thm:L3-3}}}\\
\cline{1-1}
\scriptsize{$\{0102,0112,0121\}$} &&&&\\
\cline{1-1}
\scriptsize{$\{0112,0120,0121\}$} &&&&\\
\hline
\hline

&&&&\\[-4.5mm]
\scriptsize{$\{0101,0102,0112,0120\}$} &
\multirow{3}{*}{\scriptsize{$1,2,5,11,21,36,57,85,121,\ldots$}} &
\multirow{3}{*}{\scriptsize{A050407}} &
\multirow{3}{*}{\scriptsize{
\begin{tabular}{@{}c@{}}
$1+\binom{n+1}{3}$\\
\end{tabular}}} &
\multirow{3}{*}{\scriptsize{Thm~\ref{thm:L4-3}}}\\
\cline{1-1}
\scriptsize{$\{0101,0102,0112,0121\}$} &&&&\\
\cline{1-1}
\scriptsize{$\{0101,0112,0120,0121\}$} &&&&\\
\hline

&&&&\\[-4.5mm]
\scriptsize{$\{0101,0102,0120,0121\}$} &
\multirow{2}{*}{\scriptsize{$1,2,5,11,22,42,79,149,284,\ldots$}} &
\multirow{2}{*}{\scriptsize{New}} &
\multirow{2}{*}{\scriptsize{
\begin{tabular}{@{}c@{}}
$2^{n-1}+\binom{n-1}{2}$
\end{tabular}}} &
\multirow{2}{*}{\scriptsize{Thm~\ref{thm:GT-0101-0102-0120-0121}}}\\
\cline{1-1}
\scriptsize{$\{0102,0112,0120,0121\}$} &&&&\\
\hline
\hline

&&&&\\[-4.5mm]
\scriptsize{$\{0101,0102,0112,0120,0121\}$} &
\scriptsize{$1,2,5,10,17,26,37,50,65,\ldots$} &
\scriptsize{A002522} &
\scriptsize{
\begin{tabular}{@{}c@{}}
$(n-1)^2+1$\\
\end{tabular}} &
\scriptsize{Thm~\ref{thm:A-0101-0102-0112-0120-0121}}\\
\hline

\end{tabular}%
}

\caption{A summary of our enumerative results.}\label{tab-results}
\end{table}

\section{Preliminaries}\label{sec:prelim}

A sequence $x_1 x_2 \cdots x_n$ of nonnegative integers is called a \emph{restricted growth function} (\emph{RGF}) if $x_1=0$ and, for each $i\ge2$, \[ x_i\le 1+\max\{x_1,\ldots,x_{i-1}\}. \] Equivalently, for every $k\ge1$ that occurs in $x$, the first occurrence of $k$ is preceded by an occurrence of $k-1$. In particular, it follows that this first occurrence must be preceded by occurrences of each of $0,1,\ldots,k-1$. For example, the ascent sequence $01002030$ is an RGF, while the ascent sequence $010103$ is not.  We will frequently use the following lemma, which appears as Lemma~2.4 in~\cite{DuncanSteingrimsson2011}.

\begin{lemma}[\cite{DuncanSteingrimsson2011}]\label{RGF-lem}
Let $p$ be a pattern. Then $\mathcal A_n(\{p\})$ consists solely of RGFs if and only if $p$ is a subpattern of $01012$.
\end{lemma}

The following lemma reduces certain pattern-avoidance problems to shorter patterns, allowing us to apply known structural results.

\begin{lemma}\label{RGF-lem-wilf}
If an ascent sequence $x$ is an RGF, then $x$ contains $0101$ (resp., $0102$, $0120$, $0121$) if and only if it contains $101$ (resp., $102$, $120$, $021$).
\end{lemma}

\begin{proof}
The forward implication is immediate in each case.

For $p\in\{0101,0102,0120\}$, let $p^*\in\{101,102,120\}$
be the corresponding pattern. Suppose that $x$ contains $p^*$ in
positions $i<j<k$, and let $y=\min\{x_i,x_j,x_k\}$. In each case,
$y<x_i$, so, since $x$ is an RGF, there exists $t<i$ such that
$x_t=y$. Hence $x_t x_i x_j x_k$ is order-isomorphic to $p$.

Finally, suppose that $x$ contains $021$. Then there exist indices
$i<j<k$ such that $x_i<x_k<x_j$. Since $x_k<x_j$ and $x$ is an RGF,
there exists $t<j$ such that $x_t=x_k$. Since $x_i<x_t$, there exists
$s<t$ such that $x_s=x_i$. Hence $s<t<j<k$, and
$x_s x_t x_j x_k$ is order-isomorphic to $0121$.
\end{proof}
We note that the recent work of Mansour and Shattuck
\cite{MansourShattuck} determines
$\mathcal A_n(\{021,\tau\})$ for every pattern $\tau$ of length $4$.
By Lemmas~\ref{RGF-lem} and~\ref{RGF-lem-wilf}, we have
\[
\begin{aligned}
\mathcal A_n(\{0101,0121\})
   &=\mathcal A_n(\{021,0101\}),\\
\mathcal A_n(\{0112,0121\})
   &=\mathcal A_n(\{021,0112\}),\\
\mathcal A_n(\{0102,0121\})
   &=\mathcal A_n(\{021,0102\}).
\end{aligned}
\]
Thus, these cases overlap with the results obtained in
\cite{MansourShattuck}.

For single patterns in Table~\ref{tab-results}, $0101$-avoiding ascent sequences were enumerated in \cite[Theorem~2.5]{DuncanSteingrimsson2011}, while the patterns  $0102$ and $0112$ (also $102$) were shown to be Wilf-equivalent, and enumerated, in \cite[Theorem~2.9]{DuncanSteingrimsson2011}. The enumeration of $0120$-avoiding ascent sequences and that of $0121$-avoiding ascent sequences remain open problems. 

Many of our enumerative results for the simultaneous avoidance of more than one pattern are obtained using the method of \emph{generating trees}; we refer to~\cite{BanBouDenFlaGarGou} for a detailed introduction to this method and the techniques for deriving the associated recurrence relations and generating functions.  A generating tree is a rooted tree in which each pattern-avoiding ascent sequence of length~$n$ appears exactly once at level~$n$. Such a tree is constructed by specifying suitable labels and succession rules.

Throughout the paper, for an ascent sequence $x=x_1x_2\cdots x_n$, we write $x\mapsto xt$ for the extension of $x$ by a letter $t$, and let $\max(x)$ denote the maximum letter of $x$. Whenever the forbidden set under consideration forces the RGF property, every admissible child remains an RGF. Consequently, if $xt$ is an admissible child of $x$ in the generating tree and $\max(x)=m$, then the RGF property gives $0\le t\le m+1$.

\section{Avoidance of pairs of patterns}\label{avoid-pairs-sec}

\subsection{Pairs \texorpdfstring{$\{0101,0102\}$ and $\{0101,0121\}$}{{0101,0102} and {0101,0121}}}

In the first part of the proof of the following theorem, we provide full derivational details; in subsequent proofs, some of these details will be omitted.

\begin{thm}\label{thm:L2-2}
For $n\geq 1$, we have $a_{\{0101,0102\}}(n)=a_{\{0101,0121\}}(n)=F_{2n-1}$, where $F_n$ is the $n$-th Fibonacci number defined by $F_1=F_2=1$ and, for $n \ge 3$, $F_n=F_{n-1}+F_{n-2}$.
\end{thm}
\begin{proof} 
We derive the formulas for $a_{\{0101,0102\}}(n)$ and $a_{\{0101,0121\}}(n)$ separately. 

\noindent
\textbf{Enumeration of $\mathcal{A}_n(\{0101,0102\})$}. 
By Lemmas~\ref{RGF-lem} and~\ref{RGF-lem-wilf}, it suffices to enumerate $\mathcal{A}_n(\{101,102\})$. 
To construct the generating tree, we assign each node $x$ a label of the form $(0, m)$ or $(1, \ell)$. Here, $(0,m)$ indicates that $x$ has no descents and $\max(x) = m$, while $(1,\ell)$ indicates that $x$ has at least one descent and ends in $\ell$. The root is the sequence $0$ with label $(0,0)$. The succession rules are as follows:
\[
\begin{aligned}
(0,m) & \rightsquigarrow (0,m)\ (0,m{+}1)\ (1,0)\cdots(1,m{-}1) \\
(1,\ell) & \rightsquigarrow (1,0)\cdots(1,\ell)
\end{aligned}
\]
We now justify the succession rules.

\noindent\emph{Justification of $(0,m)\rightsquigarrow (0,m)\ (0,m{+}1)\ (1,0)\cdots(1,m{-}1)$.}
Let $x$ have label $(0,m)$. By definition, $x$ has no descents and satisfies $\max(x)=m$. Since $x\mapsto xt$, we have $0\le t\le m+1$. If $0\le t\le m-1$, then $xt$ ends with the descent $mt$, and hence the child has label $(1,t)$. If $t=m$, then no descent is created and $\max(xt)=m$, so the child has label $(0,m)$. If $t=m+1$, then no descent is created and $\max(xt)=m+1$, so the child has label $(0,m+1)$.

\smallskip
\noindent\emph{Justification of $(1,\ell)\rightsquigarrow (1,0)\cdots(1,\ell)$.}
Let $x$ have label $(1,\ell)$, and let $ba$, with $b>a$, be the first (i.e., leftmost) descent in $x$. We claim that the suffix of $x$ beginning with $b$ is weakly decreasing.

To prove this, it is enough to show that no letter greater than $a$ can occur to the right of $a$. Suppose, for contradiction, that some letter $c>a$ does occur to the right of $a$. If $c>b$, then $bac$ forms an occurrence of $102$. If $c=b$, then $bac=bab$ forms an occurrence of $101$. Finally, if $a<c<b$, then, since $x$ is an RGF, the letter $c$ must already occur to the left of $b$, and thus $x$ contains a subsequence $cac$, which forms an occurrence of $101$. In either case we obtain a contradiction. Therefore, every letter to the right of $a$ is at most $a$.

Applying the same argument to every ascent occurring after the first descent, we conclude that no ascent can occur in the suffix beginning with $b$. Hence this suffix is weakly decreasing. Since $x$ ends with $\ell$, any extension $x\mapsto xt$ must satisfy $t\in\{0,1,\dots,\ell\}$, and hence yields a child labeled $(1,t)$.

\smallskip
For $m\ge 0$ and $\ell\ge 0$, let
\[
A_m(x)=\sum_{n\ge 1} a_{m,n}x^n
\qquad\text{and}\qquad
B_\ell(x)=\sum_{n\ge 1} b_{\ell,n}x^n,
\]
where $a_{m,n}$ (resp.,  $b_{\ell,n}$) is the number of nodes at level $n$ with label $(0,m)$ (resp., $(1,\ell)$). Here $x$ marks the length of the sequence.

We first derive $A_m(x)$. Clearly, $A_0(x)=\frac{x}{1-x}$, which counts ascent sequences consisting only of $0$'s. On the other hand, for $m\geq 1$, a node with label $(0,m)$ can only be produced by a parent with label $(0,m)$ or $(0,m-1)$. Therefore, in this case, $A_m(x)=xA_m(x)+xA_{m-1}(x)$.
It follows that
\[
A_m(x)=\frac{x}{1-x}A_{m-1}(x)=A_0(x)\left(\frac{x}{1-x}\right)^m 
=\left(\frac{x}{1-x}\right)^{m+1},
\]
which holds for $m \geq 0$.

We next determine $B_\ell(x)$. By the succession rules, a node with label $(1,\ell)$ arises either from a node labeled $(0,m)$ with $m-1\ge \ell$ or from a node labeled $(1,j)$ with $j\ge \ell$. We obtain
\begin{equation}\label{B-ell-101-102}
B_\ell(x)=x\left(\sum_{m\ge \ell+1}A_m(x)+\sum_{j\ge \ell}B_j(x)\right).
\end{equation}

To solve these equations, introduce the bivariate generating functions
\[
A(x,u)=\sum_{m\ge0}A_m(x)u^m
\qquad\text{and}\qquad
B(x,u)=\sum_{\ell\ge0}B_\ell(x)u^\ell.
\]
Using the explicit formula for $A_m(x)$, we obtain
\begin{equation}\label{A(x,u)-101-102}
A(x,u)=\sum_{m\ge0}\left(\frac{x}{1-x}\right)^{m+1}u^m=\frac{x}{1-x(1+u)}.
\end{equation}
Multiplying \eqref{B-ell-101-102} by $u^\ell$ and summing over all $\ell\ge0$ yields
\[
B(x,u)
=
x\sum_{\ell\ge0}\left(\sum_{m\ge\ell+1}A_m(x)\right)u^\ell
+
x\sum_{\ell\ge0}\left(\sum_{j\ge\ell}B_j(x)\right)u^\ell.
\]
Reversing the order of summation in both double sums gives
\[
\sum_{\ell\ge0}\left(\sum_{m\ge\ell+1}A_m(x)\right)u^\ell
=
\frac{A(x,1)-A(x,u)}{1-u}
\quad\text{and}\quad
\sum_{\ell\ge0}\left(\sum_{j\ge\ell}B_j(x)\right)u^\ell
=
\frac{B(x,1)-uB(x,u)}{1-u}.
\]
Hence
\[
B(x,u)=\frac{x}{1-u}\Bigl(A(x,1)-A(x,u)+B(x,1)-uB(x,u)\Bigr),
\]
and therefore
\[
B(x,u)=\frac{x\bigl(A(x,1)-A(x,u)+B(x,1)\bigr)}{1-u(1-x)}.
\]

It remains to determine $B(x,1)$. By \eqref{A(x,u)-101-102}, we have $A(x,1)=\frac{x}{1-2x}$. Applying the kernel method to the denominator $1-u(1-x)$, we set $u=\frac{1}{1-x}$ and obtain
\[
B(x,1)=A\!\left(x,\frac{1}{1-x}\right)-A(x,1)
=\frac{x(1-x)}{1-3x+x^2}-\frac{x}{1-2x}
=\frac{x^3}{(1-3x+x^2)(1-2x)}.
\]
Thus the ordinary generating function for all nodes in the tree is
\[
F(x):=A(x,1)+B(x,1)
=\frac{x}{1-2x}+\frac{x^3}{(1-3x+x^2)(1-2x)}
=\frac{x(1-x)}{1-3x+x^2}.
\]

It is straightforward to verify that $F(x)$ is the generating function for the odd-indexed Fibonacci numbers, which satisfy the recurrence relation $F_{2n-1}=3F_{2n-3}-F_{2n-5}$ with initial conditions $F_1=1$ and $F_3=2$. Hence, $a_{\{101,102\}}(n)=a_{\{0101,0102\}}(n)=F_{2n-1}$, completing the proof of the formula for $a_{\{0101,0102\}}(n)$.

\noindent
\textbf{Enumeration of $\mathcal{A}_n(\{0101,0121\})$}. Lemmas~\ref{RGF-lem} and~\ref{RGF-lem-wilf} give
$\mathcal{A}_n(\{0101,0121\})=\mathcal{A}_n(\{101,021\})$. Hence, it is enough to enumerate $\mathcal{A}_n(\{101,021\})$, which we do via a generating tree. We assign the label $(1,0)$ to every all-zero sequence. For $m\ge1$,
a node $x$ with $\max(x)=m$ is labeled $(0,m)$ if $x$ ends in $0$,
and $(1,m)$ if $x$ ends in $m$. The root is the sequence $0$,
labeled $(1,0)$. The succession rules are as follows:
\begin{equation*}
\begin{aligned}
(1,0) & \rightsquigarrow (1,0)\ (1,1) \\
(1,m) & \rightsquigarrow (0,m)\ (1,m)\ (1,m{+}1)
        \qquad (m\ge 1) \\
(0,m) & \rightsquigarrow (0,m)\ (1,m{+}1)
        \qquad (m\ge 1)
\end{aligned}
\end{equation*}
We now justify the succession rules.

\noindent\emph{Justification of $(1,0)\rightsquigarrow (1,0)\ (1,1)$.}
Let $x$ have label $(1,0)$. Then $x=0^i$ for some $i\ge1$, so any extension
$x\mapsto xt$ has $t\in\{0,1\}$. Appending $0$ yields another all-zero sequence
and hence a child with label $(1,0)$, while appending $1$ yields a child with
label $(1,1)$.

\smallskip
\noindent\emph{Justification of the rule $(1,m)\rightsquigarrow (0,m)\ (1,m)\ (1,m{+}1)$ for $m\ge 1$.}
Let $x$ have label $(1,m)$. Then $\max(x)=m$ and $x$ ends in $m$. Since $x\mapsto xt$ and $x$ is an RGF, we have $t\in\{0,1,\dots,m+1\}$. If $m\ge2$ and $1\le t\le m-1$, then, because $x$ is an RGF, the sequence $xt$ contains a subsequence $0tmt$ forming the pattern $0121$, so such letters of $t$ are forbidden. Hence the only possible letters are $t=0,m,m+1$, which yield children with labels $(0,m)$, $(1,m)$, and $(1,m+1)$, respectively. When $m=1$, the admissible letters are again exactly $0,1,2$, and a direct check shows that none of the corresponding extensions creates an occurrence of $101$ or $0121$. Thus the same succession rule holds for all $m\ge1$.

\smallskip
\noindent\emph{Justification of $(0,m)\rightsquigarrow (0,m)\ (1,m{+}1)$ for $m\ge1$.}
Let $x$ have label $(0,m)$. Then $x$ ends in $0$ and satisfies $\max(x)=m$, so any extension $x\mapsto xt$ has $t\in\{0,1,\dots,m+1\}$. As shown in the justification of $(1,m)\rightsquigarrow (0,m)\ (1,m)\ (1,m+1)$, the letters $1\le t\le m-1$ are forbidden, since they create an occurrence of $0121$. In addition, $t=m$ is forbidden, because $xt$ contains the subsequence $m0m$, whose reduction is $101$. Therefore, the only admissible letters are $t=0$ and $t=m+1$, which yield children with labels $(0,m)$ and $(1,m+1)$, respectively.

Define
\[
C_m(x)=\sum_{n\ge 1}c_{m,n}x^n \quad (m\ge 1),
\qquad
D_m(x)=\sum_{n\ge 1}d_{m,n}x^n \quad (m\ge 0),
\]
where $c_{m,n}$ (resp., $d_{m,n}$) is the number of nodes at level $n$ with label $(0,m)$ (resp., $(1,m)$). 

We first find $D_0(x)$. Since the root has label $(1,0)$ and every node with label $(1,0)$ produces exactly one child with the same label, we have $D_0(x)=x+xD_0(x)$, and hence, $D_0(x)=\frac{x}{1-x}$.

Next we determine $C_m(x)$ for $m\ge 1$. A node with label $(0,m)$ can only arise from a parent with label $(0,m)$ or $(1,m)$, and each such parent contributes exactly one child of label $(0,m)$. Therefore $C_m(x)=xC_m(x)+xD_m(x)$, so that
\begin{equation}\label{C_m-101-0121}
C_m(x)=\frac{x}{1-x}D_m(x).
\end{equation}

We now determine $D_m(x)$ for $m\ge 1$. For $m=1$, a node with label $(1,1)$ arises either from a parent with label $(1,0)$ or from a parent with label $(1,1)$. Hence $D_1(x)=xD_0(x)+xD_1(x)$, which gives
\begin{equation}\label{D_1-101-0121}
D_1(x)=\frac{x}{1-x}D_0(x)=\frac{x^2}{(1-x)^2}.
\end{equation}

Now let $m\ge 2$. A node with label $(1,m)$ may arise in exactly three ways: from a parent with label $(1,m)$, from a parent with label $(1,m-1)$, or from a parent with label $(0,m-1)$. Therefore, $D_m(x)=xD_m(x)+xD_{m-1}(x)+xC_{m-1}(x)$.
Rearranging and using (\ref{C_m-101-0121}), we obtain $(1-x)D_m(x)=x\bigl(D_{m-1}(x)+C_{m-1}(x)\bigr)
=\frac{x}{1-x}D_{m-1}(x)$, and thus, for $m\geq 2$, using induction and the initial condition~\eqref{D_1-101-0121}, we have
\[
D_m(x)=\frac{x}{(1-x)^2}D_{m-1}(x)=\frac{x^{m+1}}{(1-x)^{2m}}.
\]
It then follows from (\ref{C_m-101-0121}) that, for $m\geq 1$,
\[
C_m(x)=\frac{x^{m+2}}{(1-x)^{2m+1}}.
\]

Summing over all labels gives the ordinary generating function for the entire tree:
\[
\sum_{m\ge 1}C_m(x)+\sum_{m\ge 0}D_m(x)
=\frac{x^3}{(1-x)(1-3x+x^2)}+\left(\frac{x}{1-x}
+
\frac{x^2}{1-3x+x^2}\right)
=
\frac{x(1-x)}{1-3x+x^2}.
\]
This is again the generating function for the odd-indexed Fibonacci numbers. Hence
\[
a_{\{0101,0121\}}(n)=a_{\{101,021\}}(n)=F_{2n-1}.
\]
Our proof of Theorem~\ref{thm:L2-2} is complete. \end{proof}

\subsection{Pairs \texorpdfstring{$\{0101,0112\}$, $\{0102,0112\}$, $\{0102,0120\}$, and $\{0112,0121\}$}{{0101,0112}, {0102,0112}, {0102,0120}, and {0112,0121}}}

\begin{thm}\label{thm:L2-4}
For $n\ge 1$, we have
\begin{equation}\label{eq:L2-4}
a_{\{0101,0112\}}(n)=a_{\{0102,0112\}}(n)=a_{\{0112,0121\}}(n)=a_{\{0102,0120\}}(n)=(n-1)2^{\,n-2}+1.
\end{equation}
\end{thm}

\begin{proof}
We consider each pair of patterns separately.

\noindent
\textbf{Enumeration of $\mathcal{A}_n(\{0101,0112\})$}. By Theorem~2.9 in \cite{DuncanSteingrimsson2011}, an ascent sequence avoids $0112$ if and only if it consists of a strictly increasing sequence $012\cdots k$ followed by a weakly decreasing sequence, with an arbitrary number of $0$'s inserted in arbitrary positions. Moreover, by Theorem~2.6 in \cite{DuncanSteingrimsson2011}, ascent sequences avoiding $0101$ are precisely the RGFs corresponding to noncrossing set partitions. Combining these two characterizations, every sequence in $\mathcal A_n(\{0101,0112\})$
other than the all-zero sequence $0\cdots 0$ can be written in the form $x=\alpha_1\alpha_2\cdots \alpha_r\beta$, where $r\ge 0$ and the following conditions hold:
\begin{itemize}
\item For each $j=1,\dots,r$, the block $\alpha_j$ contains both $0$
and at least one positive entry. By the above structure for
$0112$-avoidance, its positive part is strictly increasing, while
avoidance of $0101$ prevents any previously used positive letter from
reappearing after its initial zero-run. Hence, by the RGF property,
$\alpha_j$ has the form
\begin{equation}\label{description-alpha}
\alpha_j
=
0^{p_j}(h_j+1)(h_j+2)\cdots(h_j+q_j),
\end{equation}
where $p_j,q_j\ge1$ and
$h_j=\max(\alpha_1\alpha_2\cdots\alpha_{j-1})$ for $j\ge2$, with
$h_1=0$.
\item The final block \(\beta\) is nonempty and contains at least one nonzero entry. Let \(h=\max(\alpha_1\alpha_2\cdots\alpha_r)\) when \(r\ge1\), and set \(h=0\) when \(r=0\). For convenience, we describe \(\beta\) by its reduction, so that each letter \(h+j\) is written as \(j\) (and \(0\) remains \(0\)). In particular, \(\beta\) contains the subsequence \(01\). Moreover, \(\beta\) is of the form
\begin{equation}\label{description-beta}
0^i12\cdots a,a^{i_1}(a-1)^{i_2}\cdots1^{i_a}0^{i_{a+1}},
\end{equation}
where \(i\ge1\), \(a\ge1\), and \(i_p\ge0\) for all \(p\).
\end{itemize}

We now derive the generating function for the class $\mathcal{A}_n(\{0101,0112\})$. We begin by enumerating all sequences other than the all-zero sequence $0\cdots 0$. By~\eqref{description-alpha}, once the preceding blocks are fixed,
$\alpha_j$ is uniquely determined by the two positive integers $p_j$
and $q_j$, with no additional interleaving choice. Hence
\begin{equation}\label{eq:A0112-A}
A(x):=\left(\frac{x}{1-x}\right)^2.
\end{equation}
 
To construct $\beta$, we first choose the initial string of $0$'s. Since $\beta$ begins with at least one $0$, this contributes $\sum_{p\ge 1}x^p=\frac{x}{1-x}$. Next, if the initial strictly increasing part has length $a\ge 1$, then it contributes a factor $x^a$. Once this part is fixed, the subsequent weakly decreasing part may use the letters $0,1,\dots,a$, each with arbitrary multiplicity, and hence contributes $\left(\frac{1}{1-x}\right)^{a+1}$. Summing over all $a\ge 1$, we obtain
\begin{equation}\label{eq:A0112-B}
B(x):=\left(\frac{x}{1-x}\right)\left(\sum_{a\ge 1}x^a\left(\frac{1}{1-x}\right)^{a+1}\right)=\frac{1}{1-2x}\left(\frac{x}{1-x}\right)^2.
\end{equation}

It follows that the generating function for all sequences in $\mathcal A_n(\{0101,0112\})$ other than the all-zero sequences $0\cdots 0$ is
\begin{equation}\label{eq:A0112-nonzero}
\sum_{r\ge 0}A(x)^rB(x)=\frac{B(x)}{1-A(x)}=\frac{x^2}{(1-2x)^2}.
\end{equation}
The all-zero sequences contribute $\frac{x}{1-x}$.
Hence, the generating function for
$\mathcal{A}_n(\{0101,0112\})$ is
\begin{equation}\label{F(x)-0101-0112}
F(x)=\frac{x}{1-x}+\frac{x^2}{(1-2x)^2}
=\sum_{n\ge 1}x^n+\sum_{n\ge 2}(n-1)2^{n-2}x^n,
\end{equation}
giving, for $n\geq 1$, $a_{\{0101,0112\}}(n)=1+(n-1)2^{n-2}$.
 
\noindent
\textbf{Enumeration of $\mathcal{A}_n(\{0102,0112\})$.} We use the structural description of ascent sequences avoiding $0112$ established above, and then impose the additional avoidance condition for $0102$.

Every sequence in $\mathcal{A}_n(\{0102,0112\})$ other than the all-zero sequence $0\cdots 0$ can be written in the form $x=\beta\Gamma$, the initial block $\beta$ is of the form
\begin{equation}\label{eq:beta-0102-0112}
0^i12\cdots a\, a^{i_1}(a-1)^{i_2}\cdots 1^{i_a},\qquad i\ge 1,\quad a\ge 1,\quad i_s\ge 0,
\end{equation}
and either $\Gamma=\epsilon$, or $\Gamma$ begins with $0$ and is a
word over $\{0,1\}$. Thus $\beta$ has the same general form as
in~\eqref{description-beta}, except that its weakly decreasing part
contains no $0$'s.
Indeed, let $\Gamma$ begin with the first $0$ occurring after the first
positive letter, if such a $0$ exists. If a letter $c\ge 2$ occurred
after this $0$, then the subsequence $010c$ would form an occurrence of
$0102$. Hence every subsequent letter belongs to $\{0,1\}$.
Conversely, every word of the displayed form avoids both $0102$ and
$0112$.

The generating function for the block $\beta$ may be obtained in the same way as in~\eqref{eq:A0112-B}. Since the initial string of $0$'s has positive length, it contributes $\frac{x}{1-x}$. If the initial strictly increasing part has length $a\ge 1$, then it contributes a factor $x^a$, while the subsequent weakly decreasing part may use only the letters $1,2,\dots,a$, each with arbitrary multiplicity. Hence
\begin{equation}\label{eq:B0112-B}
B(x):=\left(\frac{x}{1-x}\right)\left(\sum_{a\ge 1}x^a\left(\frac{1}{1-x}\right)^{a}\right)=\frac{x^2}{(1-x)(1-2x)}.
\end{equation}

The generating function for the suffix $\Gamma$ is obtained directly from its structure. 
The empty suffix contributes $1$, while a nonempty suffix begins with a $0$, 
contributing a factor $x$, and each subsequent letter may be chosen freely 
from $\{0,1\}$. Hence
\begin{equation}\label{eq:G-0102-0112}
G(x):=1+x\sum_{m\ge 0}(2x)^m
=1+\frac{x}{1-2x}
=\frac{1-x}{1-2x}.
\end{equation}

Finally, the all-zero sequences contribute $\frac{x}{1-x}$. Substituting \eqref{eq:B0112-B} and \eqref{eq:G-0102-0112} into
\[
F(x)=B(x)G(x)+\frac{x}{1-x},
\]
we obtain the same generating function as in \eqref{F(x)-0101-0112}. This shows that, for $n \geq 1$,
$a_{\{0102,0112\}}(n) = 1 + (n-1)2^{n-2}$.

\noindent
\textbf{Enumeration of $\mathcal{A}_n(\{0112,0121\})$.}

\begin{lemma}\label{lem:0121-0112-tree}
For a sequence $x$, let $x^+$ denote the subsequence of its positive entries. We classify sequences in $\mathcal A_n(\{0112,0121\})$ according to the form of $x^+$. For each $x\in \mathcal A_n(\{0112,0121\})$, we assign a label as follows:
\[
\begin{aligned}
(0):\; & x=0^i; \\
(01):\; & x^+=12\,\cdots\,m, \text{ where }m \ge 1;\\
(011):\; & x^+=12\,\cdots\,(m-1)\,m^s, \text{ where }m \ge 1 \text{ and }s \ge 2.
\end{aligned}
\]
Then $\mathcal{A}_n(\{0112,0121\})$ is generated by the generating tree
with root labeled $(0)$ and the following succession rules:
\[
\begin{aligned}
(0) &\rightsquigarrow (0)(01)\\
(01) &\rightsquigarrow (01)(011)(01)\\
(011) &\rightsquigarrow (011)(011)
\end{aligned}
\]
\end{lemma}

\begin{proof}
By Lemma~\ref{RGF-lem}, every sequence in $\mathcal{A}_n(\{0112,0121\})$ is an RGF, so if $m=\max(x)\ge1$, then the first occurrences of its positive letters are $1,2,\dots,m$ in order. Since avoiding $0121$ forbids any smaller positive letter after the first $m$, and avoiding $0112$ forbids any repeated positive letter before it, we obtain
\[
x^+=12\cdots m
\quad\text{or}\quad
x^+=12\cdots (m-1)m^s
\]
for some $s\ge2$.
We now justify the succession rules.

\medskip
\noindent\emph{Justification of $(0)\rightsquigarrow (0)(01)$.}
Let $x=0^i$ for some $i\ge1$. Since $x\mapsto xt$, we have $t\in\{0,1\}$. Appending $0$ (resp., $1$) yields a child with label $(0)$ (resp., $(01)$).

\medskip
\noindent\emph{Justification of $(01)\rightsquigarrow (01)(011)(01)$.}
Let $x$ have label $(01)$, so that $x^+=12\cdots m$. Since $x\mapsto xt$, we have $t\in\{0,1,\dots,m+1\}$. The letters $1\le t\le m-1$ are forbidden because $xt$ contains the subsequence $0tmt$, which forms the pattern $0121$. Hence only the letters $0$, $m$, and $m+1$ are admissible, yielding children with labels $(01)$, $(011)$, and $(01)$, respectively.

\medskip
\noindent\emph{Justification of $(011)\rightsquigarrow (011)(011)$.}
Let $x$ have label $(011)$, so that $x^+=12\cdots (m-1)m^s$ for some $s\ge2$. Since $x\mapsto xt$, we have $t\in\{0,1,\dots,m+1\}$. If $1\le t\le m-1$, then $xt$ contains $0121$, witnessed by the subsequence $0tmt$. If $t=m+1$, then $xt$ contains $0112$, witnessed by $(m-1)mm(m+1)$. Thus the only admissible letters are $t=0$ and $t=m$. In either case, the form of $x^+$ is preserved, and hence the child again has label $(011)$.

Let $z_n$, $a_n$, and $b_n$ denote the number of nodes at level $n$ with labels $(0)$, $(01)$, and $(011)$, respectively. From the generating tree, we obtain the recurrences
\[
z_n = 1,
\qquad
a_{n+1} = z_n + 2a_n,
\qquad
b_{n+1} = a_n + 2b_n,
\]
with initial values $z_1 = 1$, $a_1 = 0$, and $b_1 = 0$. Solving, for $n \geq 2$, we obtain 
\[
a_n = 2^{n-1} - 1,
\qquad
b_n = (n-3)2^{n-2} + 1.
\]
Hence, $a_{\{0112,0121\}}(n) = z_n + a_n + b_n= (n-1)2^{n-2} + 1.$
\end{proof}

\noindent
\textbf{Enumeration of $\mathcal{A}_n(\{0102,0120\})$.}
Lemmas~\ref{RGF-lem} and~\ref{RGF-lem-wilf} give
$\mathcal{A}_n(\{0102,0120\})=\mathcal{A}_n(\{102,120\})$.
The latter class was enumerated in \cite{BaxterPudwell2015}.
Thus, $a_{\{0102,0120\}}(n)=(n-1)2^{\,n-2}+1$.
This completes the proof of Theorem~\ref{thm:L2-4}.
\end{proof}

\subsection{Pair \texorpdfstring{$\{0101,0120\}$}{{0101,0120}}}

\begin{thm}\label{thm:0101-0120}
For $n\ge1$, we have $a_{\{0101,0120\}}(n)=a_{\{101,120\}}(n)=|S_n(231,4123)|$. Hence,
\[
1+\sum_{n\ge1}a_{\{0101,0120\}}(n)x^n
=
\frac{(1-x)^3}{1-4x+5x^2-3x^3}.
\]
\end{thm}

\begin{proof}
By Lemmas~\ref{RGF-lem} and~\ref{RGF-lem-wilf}, we have $\mathcal{A}_n(\{0101,0120\})=\mathcal{A}_n(\{101,120\})$. Hence, it suffices to enumerate $\mathcal{A}_n(\{101,120\})$, which was done in \cite{BaxterPudwell2015}. This completes the proof of Theorem~\ref{thm:0101-0120}. \end{proof}

\subsection{Pair \texorpdfstring{$\{0102,0121\}$}{{0102,0121}}}
\begin{thm}\label{thm:0102-0121}
For $n\ge 1$, we have $a_{\{0102,0121\}}(n)
=
\frac12\left(3\cdot 2^n-n^2-n-2\right)$.
\end{thm}
\begin{proof}
By Lemmas~\ref{RGF-lem} and~\ref{RGF-lem-wilf}, we have $\mathcal{A}_n(\{0102,0121\})=\mathcal{A}_n(\{102,021\})$. Hence, it suffices to enumerate $\mathcal{A}_n(\{102,021\})$, which was done in \cite{BaxterPudwell2015}. This completes the proof of Theorem~\ref{thm:0102-0121}. \end{proof}

\subsection{Pair \texorpdfstring{$\{0112,0120\}$}{{0112,0120}}}
\begin{thm}\label{thm:L2-0112,0120}
For $n\ge 1$, we have $a_{\{0112,0120\}}(n)=2^{\,n-1}+\binom{n+1}{4}$.
\end{thm}
\begin{proof}
For a sequence $x$, let $x^+$ denote the subsequence consisting of its positive entries. We classify the sequences in $\mathcal A_n(\{0112,0120\})$ according to the form of $x^+$. Arguing as in the proof of Lemma~\ref{lem:0121-0112-tree}, every sequence in $\mathcal A_n(\{0112,0120\})$ is an RGF, so if $x$ has a positive entry and $m=\max(x)$, then the first occurrences of its positive letters are $1,2,\dots,m$ in order. If $m\ge2$, no positive letter $r<m$ can repeat before the first occurrence of $m$, since otherwise $0rrm$ forms an occurrence of $0112$. After the first $m$, no letter $t\le m-2$ can occur, since the first occurrences of $t,t+1,m$, followed by this later $t$, form an occurrence of $0120$. Finally, once $m-1$ reappears after the first $m$, no further $m$ can occur, since $0(m-1)(m-1)m$ forms an occurrence of $0112$. Therefore, $x^+$ must be of one of the following six forms, and we assign labels accordingly:
\[
\begin{aligned}
(0):\;& x=0^i;\\
(01):\;& x^+=1;\\
(011):\;& x^+=1^p \text{ for some } p\ge2;\\
(012):\;& x^+=12\cdots m \text{ for some } m\ge2;\\
(0122):\;& x^+=12\cdots (m-1)m^p \text{ for some } m\ge2,\ p\ge2;\\
(0121):\;& x^+=12\cdots (m-1)m^p(m-1)^q \text{ for some } m\ge2,\ p\ge1,\ q\ge1.
\end{aligned}
\]

With these labels, $\mathcal A_n(\{0112,0120\})$ is generated by a finite generating tree with root $(0)$ and the following succession rules:
\[
\begin{aligned}
\begin{array}{l@{\qquad}l}
\begin{aligned}
(0)    &\rightsquigarrow (0)(01)\\
(01)   &\rightsquigarrow (01)(011)(012)\\
(011)  &\rightsquigarrow (011)(011)
\end{aligned}
&
\begin{aligned}
(012)  &\rightsquigarrow (0121)(0122)(012)\\
(0122) &\rightsquigarrow (0121)(0122)\\
(0121) &\rightsquigarrow (0121)
\end{aligned}
\end{array}
\end{aligned}
\]
The first three rules are proved as in Lemma~\ref{lem:0121-0112-tree}. We now justify the remaining ones.

We first record a common observation. Let $x$ be a sequence with label $(012)$, $(0122)$, or $(0121)$. Then any extension $xt$ with $t\le m-2$ creates an occurrence of $0120$. Indeed, in each case, $x^+$ contains $12\cdots m$, so appending such a $t$ yields a subsequence order-isomorphic to $0120$.

If $x$ has label $(012)$, then the admissible letters of $t$ are $m-1$, $m$, and $m+1$, giving rise to children with labels $(0121)$, $(0122)$, and $(012)$, respectively. Hence, $(012)\rightsquigarrow(0121)(0122)(012)$.

If $x$ has label $(0122)$, then $t=m+1$ creates $0112$, so the admissible letters of $t$ are $m-1$ and $m$, yielding children with labels $(0121)$ and $(0122)$, respectively. Hence, $(0122)\rightsquigarrow(0121)(0122)$.

If $x$ has label $(0121)$, then any $t\ge m$ creates $0112$, so the only admissible value of $t$ is $m-1$. Hence, $(0121)\rightsquigarrow(0121)$.

Now let $z_n$, $a_n$, $b_n$, $c_n$, $d_n$, and $e_n$ denote the numbers of nodes at level $n$
with labels $(0)$, $(01)$, $(011)$, $(012)$, $(0122)$, and $(0121)$, respectively.
Then $a_1=b_1=c_1=d_1=e_1=0$ and $z_1=1$, and the succession rules imply
\[
\begin{aligned}
z_{n+1}&=z_n, & a_{n+1}&=z_n+a_n, & b_{n+1}&=a_n+2b_n,\\
c_{n+1}&=a_n+c_n, & d_{n+1}&=c_n+d_n, & e_{n+1}&=c_n+d_n+e_n.
\end{aligned}
\]

Solving these recurrences yields
\[
\begin{aligned}
z_n&=1, & a_n&=n-1, & b_n&=2^{\,n-1}-n,\\
c_n&=\binom{n-1}{2}, & d_n&=\binom{n-1}{3}, & e_n&=\binom{n}{4}.
\end{aligned}
\]
Therefore, $a_{\{0112,0120\}}(n)=z_n+a_n+b_n+c_n+d_n+e_n=2^{n-1}+\binom{n+1}{4}$, as desired.\end{proof}

\section{Avoidance of triples of patterns}\label{avoid-triple-sec}

\subsection{\texorpdfstring{Triple $\{0101,0112,0120\}$}{{0101,0112,0120}}}

\begin{thm}\label{thm:L3-1}
We have $a_{\{0101,0112,0120\}}(1)=1$, and for $n\ge2$,
\begin{equation}\label{formula-0101,0112,0120}
a_{\{0101,0112,0120\}}(n)
=
\frac1{24}\left(n^4-6n^3+47n^2-114n+120\right).
\end{equation}
\end{thm}

\begin{proof}
For $n=2$, the only ascent sequences are $00$ and $01$.
Both avoid $0101$, $0112$, and $0120$.
Hence, $a_{\{0101,0112,0120\}}(2)=2$,
in agreement with~\eqref{formula-0101,0112,0120}.
We therefore assume $n\ge3$ below.

We obtain this class by refining the generating tree for
$\mathcal A_n(\{0112,0120\})$ in the proof of Theorem~\ref{thm:L2-0112,0120}, while the labels descended from $(01)$ must be refined in order to account
for the additional avoidance of $0101$.  
The resulting succession rules are as follows:
\[
\begin{aligned}
\begin{array}{l@{\qquad}l}
\begin{aligned}
(0)      &\rightsquigarrow (0)(01)\\
(01)     &\rightsquigarrow (010)(011)(012)\\
(011)    &\rightsquigarrow (0110)(011)\\
(010)    &\rightsquigarrow (010)(0102)\\
(012)    &\rightsquigarrow (0121)(0122)(012)
\end{aligned}
&
\begin{aligned}
(0110)   &\rightsquigarrow (0110)\\
(0102)   &\rightsquigarrow (012)(01022)\\
(0121)   &\rightsquigarrow (0121)\\
(0122)   &\rightsquigarrow (0121)(0122)\\
(01022)  &\rightsquigarrow (01022)
\end{aligned}
\end{array}
\end{aligned}
\]
We now justify the succession rules.

\medskip
\noindent\emph{Justification of $(0)\rightsquigarrow (0)(01)$.}
The root $0$ is labeled $(0)$, and its two children are $00$ and $01$. Since every all-zero ascent sequence has an isomorphic set of children, all and only such sequences receive label $(0)$; the other child is labeled $(01)$.

\medskip
\noindent\emph{Justification of $(01)\rightsquigarrow (010)(011)(012)$.}
The first sequence with label $(01)$ is $01$, whose children are $010$, $011$, and $012$. More generally, any sequence with label $(01)$ has exactly three admissible extensions: appending $0$, repeating the last letter, or appending a new maximum. These yield the labels $(010)$, $(011)$, and $(012)$, respectively.

\medskip
\noindent\emph{Justification of $(010)\rightsquigarrow (010)(0102)$.}
The children of $010$ are $0100$, $0101$, and $0102$, but $0101$ is forbidden because it contains the pattern $0101$. Thus $0100$ retains label $(010)$, while $0102$ receives label $(0102)$. More generally, if $x$ has label $(010)$, then its only admissible extensions are obtained by repeating the last letter or by appending $\max(x)+1$, and these children receive labels $(010)$ and $(0102)$, respectively.

\medskip
\noindent\emph{Justification of $(011)\rightsquigarrow (0110)(011)$.}
The children of $011$ are $0110$, $0111$, and $0112$, but $0112$ is forbidden because it contains the pattern $0112$. Hence the first two children receive labels $(0110)$ and $(011)$, respectively. More generally, any sequence with label $(011)$ has exactly two admissible extensions: appending $0$ or repeating the last letter, yielding labels $(0110)$ and $(011)$.

\medskip
\noindent\emph{Justification of $(012)\rightsquigarrow (0121)(0122)(012)$.}
The first sequence with label $(012)$ is $012$, whose children are $0120$, $0121$, $0122$, and $0123$. The first is forbidden because it contains the pattern $0120$. For $0123$, once the letter $3$ appears, no further $0$ can occur; otherwise, the sequence would contain the pattern $0120$. Thus the initial $0$ becomes irrelevant, and only the active suffix $123$ matters for subsequent pattern avoidance. We therefore assign $0123$ the same label $(012)$. More generally, if $x$ has label $(012)$, then its only admissible extensions are obtained by appending $\max(x)-1$, $\max(x)$, or $\max(x)+1$, and these children receive labels $(0121)$, $(0122)$, and $(012)$, respectively.

\medskip
\noindent\emph{Justification of $(0122)\rightsquigarrow (0121)(0122)$.}
The children of $0122$ are $01220$, $01221$, $01222$, and $01223$. The first and last are forbidden, since they contain the patterns $0120$ and $0112$, respectively. Hence the admissible extensions are obtained by appending $\max(x)-1$ or $\max(x)$, and these yield the labels $(0121)$ and $(0122)$.

\medskip
\noindent\emph{Justification of $(0102)\rightsquigarrow (012)(01022)$.}
The children of $0102$ are $01020$, $01021$, $01022$, and $01023$. The first two are forbidden, since they contain the patterns $0120$ and $0101$, respectively. The child $01022$ receives label $(01022)$. For $01023$, once the letter $3$ appears, no further occurrence of $0$ or $1$ is allowed: a later $0$ would create the pattern $0120$, while a later $1$ would create the pattern $0101$. Thus, we assign $01023$ the label $(012)$. More generally, if $x$ has label $(0102)$, then its only admissible extensions are obtained by repeating the last letter or by appending $\max(x)+1$, and these children receive labels $(01022)$ and $(012)$, respectively. 

\medskip
\noindent\emph{Justification of $(01022)\rightsquigarrow (01022)$.}
The children of $01022$ are $010220$, $010221$, $010222$, and $010223$. The first, second, and fourth are forbidden, since they contain the patterns $0120$, $0101$, and $0112$, respectively. Thus only $010222$ is admissible, and it remains in the same class. More generally, every sequence with label $(01022)$ is of the form $0\cdots010\cdots02\cdots2$, and its unique admissible child is obtained by repeating the last letter. 
Similarly, one obtains $(0110)\rightsquigarrow (0110)$ and $ (0121)\rightsquigarrow (0121).$

Now let $z_n$, $a_n$, $b_n$, $c_n$, $d_n$, $e_n$, $f_n$, $g_n$, $h_n$, $i_n$ denote the numbers of nodes at level $n$ carrying the labels (0), (01), (011), (010), (012), (0110), (0102), (0121), (0122), (01022), respectively. The succession rules give the system of equations: $z_{n+1}=z_n$,
 $a_{n+1}=z_n$,  $b_{n+1}=a_n+b_n$, $c_{n+1}=a_n+c_n$, $d_{n+1}=a_n+d_n+f_n$, $e_{n+1}=b_n+e_n$, $f_{n+1}=c_n$, $g_{n+1}=d_n+g_n+h_n$, $h_{n+1}=d_n+h_n$, and $i_{n+1}=f_n+i_n$, 
with initial values $z_1=1$ and $a_1=b_1=c_1=d_1=e_1=f_1=g_1=h_1=i_1=0$. For $n\ge3$, solving these recurrences yields
\[
\begin{aligned}
z_n&=1,\; a_n=1,\; b_n=n-2,\; c_n=n-2,\;
d_n=1+\binom{n-2}{2},\; e_n=\binom{n-2}{2},\\
f_n&=n-3,\;
g_n=\binom{n-2}{2}+\binom{n-2}{3}+\binom{n-2}{4},\;
h_n=(n-3)+\binom{n-2}{3},\;
i_n=\binom{n-3}{2}.
\end{aligned}
\]
By simplifying $z_n+a_n+b_n+c_n+d_n+e_n+f_n+g_n+h_n+i_n$, we obtain~\eqref{formula-0101,0112,0120}.
\end{proof}

\begin{thm}\label{thm:L3-3}
For $n\ge1$, we have
\[
a_{\{0102,0112,0120\}}(n)
=
a_{\{0102,0112,0121\}}(n)
=
a_{\{0112,0120,0121\}}(n)
=
2^{n-1}+\binom{n}{3}.
\]
\end{thm}

\begin{proof} We consider each triple of patterns separately.

\noindent
\textbf{Enumeration of $\mathcal A_n(\{0102,0112,0120\})$.}
We begin with the generating tree for $\mathcal{A}_n(\{0112,0120\})$ and impose the additional avoidance of $0102$. Arguing as in the proof of Theorem~\ref{thm:L3-1}, we obtain the following succession rules:
\[
\begin{aligned}
\begin{array}{l@{\qquad}l}
\begin{aligned}
(0)     &\rightsquigarrow (0)(01)\\
(01)    &\rightsquigarrow (010)(011)(012)\\
(011)   &\rightsquigarrow (010)(011)\\
(010)   &\rightsquigarrow (010)(010)
\end{aligned}
&
\begin{aligned}
(012)   &\rightsquigarrow (0121)(0122)(012)\\
(0121)  &\rightsquigarrow (0121)\\
(0122)  &\rightsquigarrow (0121)(0122)
\end{aligned}
\end{array}
\end{aligned}
\]
The labels $(0)$, $(01)$, $(012)$, $(0121)$, and $(0122)$ obey the same succession rules as in the proof of Theorem~\ref{thm:L3-1}, except that branches creating $0102$ are deleted. Since $0101$ is no longer forbidden, a node with label $(010)$ has two admissible children, both labeled $(010)$, while a node with label $(011)$ has two admissible children with labels $(010)$ and $(011)$.

Now let $z_n$, $a_n$, $b_n$, $c_n$, $d_n$, $e_n$, and $f_n$ denote the numbers of nodes at level $n$ carrying the labels $(0)$, $(01)$, $(011)$, $(010)$, $(012)$, $(0121)$, and $(0122)$, respectively. The above succession rules give
\[
\begin{array}{l@{\qquad}l@{\qquad}l}
z_{n+1}=z_n,      & a_{n+1}=z_n,            & b_{n+1}=a_n+b_n,\\
c_{n+1}=a_n+b_n+2c_n, & d_{n+1}=a_n+d_n,   & e_{n+1}=d_n+e_n+f_n,\\
f_{n+1}=d_n+f_n.  &                         &
\end{array}
\]
with initial values $z_1=1$ and $a_1=b_1=c_1=d_1=e_1=f_1=0$. Thus, for $n\ge2$,
\[
z_n=1,\ a_n=1,\ b_n=n-2,\ c_n=2^{\,n-1}-n,\ d_n=n-2,\ e_n=\binom{n-1}{3},\ f_n=\binom{n-2}{2},
\]
and hence $a_{\{0102,0112,0120\}}(n)=z_n+a_n+b_n+c_n+d_n+e_n+f_n=2^{\,n-1}+\binom{n}{3}$.

\medskip
\noindent
\textbf{Enumeration of $\mathcal A_n(\{0102,0112,0121\})$.}
The two classes have isomorphic generating trees, with the same
succession rules obtained by interchanging the roles of the forbidden
patterns $0120$ and $0121$ under the corresponding relabeling.
Hence,
$a_{\{0102,0112,0121\}}(n)=a_{\{0102,0112,0120\}}(n)$.

\medskip
\noindent
\textbf{Enumeration of $\mathcal A_n(\{0112,0120,0121\})$.} 
The generating tree for this class is obtained directly from that for
$\mathcal A_n(\{0112,0120\})$ by further imposing avoidance of $0121$,
which deletes the branches creating an occurrence of $0121$.
All other labels and succession rules remain unchanged. Therefore, by the enumeration in Theorem~\ref{thm:L2-0112,0120}, we have $a_{\{0112,0120,0121\}}(n)=z_n+a_n+b_n+c_n+d_n=2^{n-1}+\binom{n}{3}$, as desired. 

This completes the proof of Theorem~\ref{thm:L3-3}.
\end{proof}

\subsection{\texorpdfstring
{Triples $\{0101,0102,0112\}$,
$\{0101,0102,0120\}$,
$\{0101,0102,0121\}$,\\
$\{0101,0112,0121\}$,
$\{0101,0120,0121\}$ and
$\{0102,0120,0121\}$}
{{0101,0102,0112}, {0101,0102,0120},
{0101,0102,0121}, {0101,0112,0121},
{0101,0120,0121} and {0102,0120,0121}}}

\begin{thm}\label{thm:L3-6}
For $n\geq 1$, we have
\begin{equation}\label{eq:L3-6}
\begin{aligned}
&a_{\{0101,0102,0112\}}(n)=a_{\{0101,0102,0120\}}(n)=a_{\{0101,0102,0121\}}(n)=\\
&a_{\{0101,0112,0121\}}(n)=a_{\{0101,0120,0121\}}(n)=a_{\{0102,0120,0121\}}(n)=2^n-n.
\end{aligned}
\end{equation}
\end{thm}

\begin{proof} We divide the proof into three parts.

\noindent
\textbf{Enumeration of $\mathcal A_n(\{0101, 0102, 0112\})$.} 
We use the structural description of sequences avoiding $\{0102,0112\}$ and impose avoidance of $0101$.
In the proof of Theorem~\ref{thm:L2-4}, the suffix $\Gamma$ is either empty or begins with $0$, with all subsequent letters in $\{0,1\}$. Under the additional avoidance of $0101$, $\Gamma$ can contain no $1$, and hence $\Gamma=0^k$ for some $k\ge0$. Thus the suffix contributes $\frac{1}{1-x}$.
Consequently, by applying the enumeration results from Theorem~\ref{thm:L2-4}, we obtain
\[
F(x) =B(x)\frac{1}{1-x}+\frac{x}{1-x} 
       = \frac{x^2}{(1-x)^2(1-2x)} + \frac{x}{1-x}
       =\frac{x(1-2x+2x^2)}{(1-x)^2(1-2x)},
\]
whose coefficients are given explicitly by $a_{\{0101,0102,0112\}}(n) = 2^n - n$.

\medskip
\noindent
\textbf{Enumeration of $\mathcal A_n(\{0101,0112,0121\})$.}
By the structural decomposition of $\mathcal A_n(\{0101,0112\})$ established in the proof of Theorem~\ref{thm:L2-4}, it remains only to modify the contribution of the final block $\beta$ in order to impose the additional avoidance of $0121$. In the case of $\mathcal A_n(\{0101,0112\})$, the weakly decreasing part of $\beta$ may involve the letters $a,a-1,\dots,1,0$. The additional restriction $0121$ excludes every letter $t$ with $1\le t\le a-1$, since the resulting sequence would contain the subsequence $0tat$, which forms the pattern $0121$. Hence the weakly decreasing part of $\beta$ can involve only the letters $a$ and $0$. Therefore, the generating function for $\beta$ becomes
\begin{equation}\label{BBB}
B(x)=\left(\frac{x}{1-x}\right)\left(\sum_{a\ge1}x^a\left(\frac{1}{1-x}\right)^2\right)=\frac{x^2}{(1-x)^4}.
\end{equation}
By \eqref{eq:A0112-A} and \eqref{BBB}, the nonzero sequences in
$\mathcal A_n(\{0101,0112,0121\})$ have generating function
\[
\sum_{r\ge0}A(x)^rB(x)
=\frac{B(x)}{1-A(x)}
=\frac{x^2}{(1-x)^2(1-2x)}.
\]
Adding the contribution of the all-zero sequences, namely $\frac{x}{1-x}$, we obtain the same generating function as $\mathcal A_n(\{0101,0102,0112\})$. Therefore, for $n \geq 1$, we have $a_{\{0101,0112,0121\}}(n)=2^n-n$.

\medskip
\noindent
\textbf{Enumeration of $\mathcal A_n(\{0101,0102,0120\})$, $\mathcal A_n(\{0101,0102,0121\})$, $\mathcal A_n(\{0101,0120,0121\})$, \\ and $\mathcal A_n(\{0102,0120,0121\})$}. By Lemmas~\ref{RGF-lem} and~\ref{RGF-lem-wilf}, for all $n\ge 1$,
\[
\begin{aligned}
a_{\{0101,0102,0120\}}(n) &= a_{\{101,102,120\}}(n), &
a_{\{0101,0102,0121\}}(n) &= a_{\{101,102,021\}}(n),\\
a_{\{0101,0120,0121\}}(n) &= a_{\{101,120,021\}}(n), &
a_{\{0102,0120,0121\}}(n) &= a_{\{102,120,021\}}(n),
\end{aligned}
\]
so all four classes belong to Class~41 in \cite[Table~1]{CallanMansour2025}, and their enumerations follow directly from the results there. The proof of Theorem~\ref{thm:L3-6} is complete. \end{proof}

\section{Simultaneous avoidance of four and five patterns}\label{four-sec}
\subsection{Sets \texorpdfstring
{$\{0101,0102,0112,0120\}$,
 $\{0101,0102,0112,0121\}$\\
 and $\{0101,0112,0120,0121\}$}
{{0101,0102,0112,0120}, {0101,0102,0112,0121}
 and {0101,0112,0120,0121}}}

\begin{thm}\label{thm:L4-3}
For $n\geq 1$, we have
\begin{equation*}
a_{\{0101,0102,0112,0120\}}(n)=a_{\{0101,0102,0112,0121\}}(n)=a_{\{0101,0112,0120,0121\}}(n)=1+\binom{n+1}{3}.
\end{equation*}
\end{thm}

\begin{proof} We divide the proof into two parts.

\noindent
\textbf{Enumeration of $\mathcal A_n(\{0101,0102,0112,0120\})$ and $\mathcal A_n(\{0101,0112,0120,0121\})$.}
The generating tree for $\mathcal A_n(\{0101,0102,0112,0120\})$ (resp., $\mathcal A_n(\{0101,0112,0120,0121\})$) is obtained from that for $\mathcal A_n(\{0101,0112,0120\})$ by further imposing avoidance of $0102$ (resp., $0121$), that is, by deleting the branches that create the corresponding forbidden pattern. Thus, in the recurrences derived in the proof of Theorem~\ref{thm:L3-1}, the terms $f_n$ and $i_n$ (resp., $g_n$) disappear, and the only remaining change is that $d_{n+1}=a_n+d_n$ (resp., no change). Consequently, \(a_{\{0101,0102,0112,0120\}}(n)=a_{\{0101,0112,0120,0121\}}(n)=1+\binom{n+1}{3}\).

\medskip
\noindent
\textbf{Enumeration of $\mathcal A_n(\{0101,0102,0112,0121\})$.} Since the two classes have isomorphic generating trees, with the same
succession rules obtained by interchanging the roles of the forbidden
patterns $0120$ and $0121$ under the corresponding relabeling, we conclude that
$a_{\{0101,0102,0112,0121\}}(n)
=a_{\{0101,0102,0112,0120\}}(n)$. The proof of Theorem~\ref{thm:L4-3} is complete. 
\end{proof}

\subsection{Sets \texorpdfstring{$\{0101,0102,0120,0121\}$ and $\{0102,0112,0120,0121\}$}{{0101,0102,0120,0121} and {0102,0112,0120,0121}}}

\begin{thm}\label{thm:GT-0101-0102-0120-0121}
For $n\ge1$, $a_{\{0101,0102,0120,0121\}}(n)=a_{\{0102,0112,0120,0121\}}(n)=2^{\,n-1}+\binom{n-1}{2}$.
\end{thm}

\begin{proof} We divide the proof into two parts.

\noindent
\textbf{Enumeration of $\mathcal A_n(\{0101,0102,0120,0121\})$.} By Lemmas~\ref{RGF-lem} and~\ref{RGF-lem-wilf}, we have the following equality
$\mathcal A_n(\{0101,0102,0120,0121\})=\mathcal A_n(\{101,102,120,021\})$.
The enumeration of this class appears as Class~52 of Table~4
in \cite{CallanMansour2025Quadruple}, with its generating
function given in Table~1. For completeness, we give
a generating-tree proof below.
We obtain the generating tree for this class from that for $\mathcal A_n(\{101,102,120\})$, given in Class~41 of Table~1 in \cite{CallanMansour2025}, by further imposing avoidance of $021$. Thus the root label and succession rules are as follows:
\begin{equation*}
\begin{array}{l@{\qquad}l}
(0)   \rightsquigarrow (0)(01)
& (01)  \rightsquigarrow (010)(01)(012)\\
(010) \rightsquigarrow (010)
& (012) \rightsquigarrow (012)(012)
\end{array}
\end{equation*}

Indeed, if $x$ has label $(01)$, then $x=0^i1^j$ for some $i,j\ge1$, and appending $0$, $1$, or $2$ yields children with labels $(010)$, $(01)$, and $(012)$, respectively. If $x$ has label $(010)$, then $x=0^i1^j0^k$ for some $i,j,k\ge1$, and only $0$ can be appended, since appending $1$ (resp., $2$) creates $101$ (resp., $102$). If $x$ has label $(012)$, then $x=0^{i_0}1^{i_1}2^{i_2}\cdots m^{i_m}$,\ $i_0,i_1,\dots,\ i_m\geq1$,$m\ge2$, and since the class avoids $120$ and $021$, appending $0$ or any letter in $\{1,\dots,m-1\}$ is forbidden, while appending $m$ or $m+1$ yields two children with the same label $(012)$.

Now let $z_n$, $a_n$, $b_n$, and $c_n$ denote the numbers of nodes at level $n$ carrying the labels $(0)$, $(01)$, $(010)$, and $(012)$, respectively. Then
\begin{equation*}
z_{n+1}=z_n,\quad
a_{n+1}=z_n+a_n,\quad
b_{n+1}=a_n+b_n,\quad
c_{n+1}=a_n+2c_n,
\end{equation*}
with initial values $z_1=1$ and $a_1=b_1=c_1=0$. Solving these recurrences gives
\begin{equation*}
z_n=1,\qquad
a_n=n-1,\qquad
b_n=\binom{n-1}{2},\qquad
c_n=2^{\,n-1}-n.
\end{equation*}
Therefore,
$
a_{\{0101,0102,0120,0121\}}(n)=z_n+a_n+b_n+c_n=2^{\,n-1}+\binom{n-1}{2}.
$
This completes the proof.

\medskip
\noindent
\textbf{Enumeration of $\mathcal A_n(\{0102,0112,0120,0121\})$.}
We begin with the generating tree already constructed for the set
$\mathcal{A}_n(\{0102,0112,0120\})$
and further impose avoidance of $0121$. Thus, in the recurrences obtained in the proof of Theorem~\ref{thm:L3-3}, the term $e_n$ is removed. It follows that $a_{\{0102,0112,0120,0121\}}(n)=2^{\,n-1}+\binom{n-1}{2}$. This completes the proof of Theorem~\ref{thm:GT-0101-0102-0120-0121}.
\end{proof}

\subsection{Set \texorpdfstring{$\{0101,0102,0112,0120,0121\}$}{{0101,0102,0112,0120,0121}}}

\begin{thm}\label{thm:A-0101-0102-0112-0120-0121}
For $n\ge1$, we have $a_{\{0101,0102,0112,0120,0121\}}(n)=(n-1)^2+1$.
\end{thm}
\begin{proof}
Starting from the generating tree for $\mathcal A_n(\{0101,0112,0120,0121\})$ obtained in the proof of Theorem~\ref{thm:L4-3}, we further impose avoidance of $0102$ by deleting the branches that create an occurrence of $0102$.
Thus, in the recurrences derived in the proof of Theorem~\ref{thm:L3-1},
the terms $f_n$, $g_n$, and $i_n$ disappear, and the only remaining
change is that $d_{n+1}=a_n+d_n$. For $n\ge2$,
\[
z_n=a_n=1,\qquad
b_n=c_n=d_n=n-2,\qquad
e_n=h_n=\binom{n-2}{2}.
\]
Therefore,
\(a_{\{0101,0102,0112,0120,0121\}}(n)
=
2+3(n-2)+2\binom{n-2}{2}
=
(n-1)^2+1.\)
This completes the proof of Theorem~\ref{thm:A-0101-0102-0112-0120-0121}.\end{proof}

\section{Concluding remarks}\label{concluding}

We have investigated the avoidance of subsets of the five patterns
$0101$, $0102$, $0112$, $0120$, and $0121$ in ascent sequences, obtaining a nearly complete enumerative picture and identifying $16$ Wilf equivalence classes. Several natural problems remain open. In particular, the enumeration of $0120$-avoiding and $0121$-avoiding ascent sequences, as well as their simultaneous avoidance, remains unresolved and appears to require new ideas.

The connections with other Fishburn objects also suggest several directions for further study. In particular, for the classes involving
$0121$ that also avoid one of $0101$, $0102$, or $0112$,
Lemmas~\ref{RGF-lem} and~\ref{RGF-lem-wilf} identify the avoidance
of $0121$ with the avoidance of $021$, connecting these classes with
the work of Mansour and Shattuck~\cite{MansourShattuck}.
Moreover, the transport-of-patterns framework of Cerbai and
Claesson~\cite{CerbaiClaesson2023}, together with related work on
pattern avoidance in modified ascent sequences~\cite{Cerbai2025}
and on Fishburn permutations and ascent sequences~\cite{Egge2022},
provides useful context for seeking bijective explanations across
different Fishburn families.

Finally, Table~\ref{tab-results} indicates numerous connections to sequences in the OEIS~\cite{OEIS}, which collectively encode many interesting combinatorial objects. Explaining the most intriguing of these connections bijectively, by relating our pattern-avoiding ascent sequences to such objects, is an interesting direction for future research.

\acknowledgements 
The work of Philip B. Zhang was supported by the Natural Science Foundation of Tianjin Municipal (No.\ 25JCYBJC00430).

\bibliographystyle{authorfullrefs}
\bibliography{references}

\end{document}